\documentclass[11pt]{amsart}

\usepackage[utf8]{inputenc}
\usepackage{paralist}
\usepackage[colorlinks=true,linkcolor=blue,urlcolor=red,citecolor=blue]{hyperref}
\usepackage{amsfonts,amsmath,amssymb,amsthm}
\usepackage{geometry}
\usepackage{listings,color,pgf}
\usepackage{tikz}
\usepackage{pgfplots}
\definecolor{colKeys}{rgb}{0,0,1} 
\definecolor{colIdentifier}{rgb}{0,0,0} 
\definecolor{colComments}{rgb}{0,1,0.3} 
\definecolor{colString}{rgb}{0,0.5,0} 

\definecolor{dkgreen}{rgb}{0,0.6,0} 
\definecolor{gray}{rgb}{0.5,0.5,0.5} 
\definecolor{lightgray}{rgb}{0.9,0.9,0.9} 

\usepackage{float}   
\newfloat{Code}{H}{myc}
   
\newcommand{\field}[1]{\mathbb{#1}}
\newcommand{\R}{\field{R}}
\newcommand{\N}{\field{N}}

\newcommand{\cA}{{\mathcal A}}

\newcommand{\cW}{\mathcal{W}}

\newcommand{\eps}{{\varepsilon}}

\newcommand{\argmin}{\operatorname{argmin}}

\newcommand{\operatorspan}{\operatorname{span}}
\newcommand{\supp}{\operatorname{supp}}
\newcommand{\diam}{\operatorname{diam}}

\newtheorem{lemma}{Lemma}
\newtheorem{theorem}{Theorem}
\newtheorem{corollary}{Corollary}

\newtheorem{proposition}{Proposition}

\title[Dynamic programming using radial basis functions]{Dynamic programming\\ using radial basis functions}

\author[Oliver Junge, Alex Schreiber]{ }

\email{oj@tum.de}
\email{schreiber@ma.tum.de}

\thanks{Research has been partially supported  by the EU Marie Curie initial training
    network \emph{Sensitivity analysis for deterministic controller design (SADCO)} in FP 7.\\ \today}
    
\begin{document}

\maketitle   

\centerline{\scshape Oliver Junge and Alex Schreiber}
\medskip
{\footnotesize
 \centerline{Center for Mathematics}
   \centerline{Technische Universit\"at M\"unchen}
   \centerline{85747 Garching bei M\"unchen}
}

\begin{abstract}
We propose a discretization of the optimality principle in dynamic programming based on radial basis functions and Shepard's moving least squares approximation method.  We prove convergence of the approximate optimal value function to the true one and present  several numerical experiments.
\end{abstract}

\section{Introduction}

For many optimal control problems, solutions can elegantly be characterized and computed by \emph{dynamic programming},  i.e.\ by solving a fixed point equation, the \emph{Bellman equation}, for the optimal value function of the problem.  For every point in the state space of the underlying problem, this function yields the optimal cost associated to this initial condition.  At the same time, the value function allows to construct an optimal controller in feedback form, enabling a robust stabilization of a possibly unstable nominal system.  This controller not only yields optimal cost trajectories, but also a maximal domain of stability for the closed loop system.  

The generality and flexibility of dynamic programming, however, comes at a price:  Since in many cases, closed form solutions of the optimality principle are not available, numerical approximations have to be sought.  This typically hinders the treatment of problems with higher dimensional state spaces, since the numerical effort scales exponentially in its dimension (Bellman's \emph{curse of dimension}). Early works on numerical schemes for the related Hamilton-Jacobi(-Bellman) equations \cite{Ca-Do83a,BaCa-Do97a,Fa87a,Ca-DoFa89a} were based on finite difference or finite element type space discretizations with interpolation type projection operators. Based on these, higher order \cite{FaFe94a,FaFe95a} and adaptive schemes \cite{Gr97a,Gr04a} have been developed, as well as interpolation based approaches for certain systems in dimensions 3-5 \cite{CaFaFe04a}.  While in these works a simple (``value'') iteration is used in order to solve the fixed point problem, one-pass methods reminiscent of Dijkstra's shortest path algorithm can be employed in order to solve the discrete problem more efficiently \cite{SeVl03a,KaOsTs05a,JuOs04a,GrJu05a}.  Similar savings can be obtained by exploiting the fact that the problem becomes linear in the max-plus algebra \cite{McEn06a}.  

In \cite{Cecil2004327,Huang:2006wu,Alwardi2012305}, radial basis functions have been proposed in order to space-discretize time-dependent Hamilton-Jacobi(-Bellman) equations by collocation. The appealing feature of this ``meshfree'' approach is its simplicity: The discretization is given by a set (e.g.\ a grid) of points in phase space which serve as the centers for the basis functions.  In comparison to, e.g., finite element methods, no additional geometric information has to be computed.
Note, however, that per se this does not avoid the curse of dimension.

In this manuscript, we consider a discrete time optimal control problem (which can be obtained from some HJB equation by a discretization in time with fixed time step) and show that using radial basis functions in combination with a  least squares type projection (aka \emph{Shepard's method}) one obtains a simple, yet general scheme for the numerical solution of the problem which also allows for a simple convergence theory (which is not given in the above mentioned works).  After a brief review on radial basis functions in Section~\ref{sec:RBF}, we prove convergence of the approximate value function to the true one as the fill distance of the discrete node set goes to zero (Section~\ref{chapRBFshepard}) and consider several numerical experiments (Section~\ref{sec:experiments}).  The Matlab codes for these examples can be downloaded from the homepages of the authors.

\section{Problem statement}\label{sec:basics}

We consider a discrete time control system
\begin{equation}\label{eq:control_system}
x_{k+1}=f(x_k,u_k), \quad k=0,1,2,\ldots,
\end{equation}
with a continuous map $f: \Omega \times U \to \Omega$ on compact sets $\Omega \subset \R^s$, $U \subset \R^d$, $0\in U$, as phase and control space, respectively.  In addition, we are given a continuous \emph{cost function} $c: \Omega \times U \to [0,\infty)$ and a compact \emph{target set} $T\subset\Omega$ and we assume that $c(x,u)$ is bounded from below by a constant $\delta > 0$ for $x\not\in T$ and all $u\in U$.

Our goal is to design a \emph{feedback law} $F:S\to U$, $S\subset\Omega$,
that stabilizes the system in the sense that discrete trajectories  for the \emph{closed loop system} $x_{k+1}=f(x_k,F(x_k)), k=0,1,2,\ldots,$ reach $T$ in a finite number of steps for starting points in a maximal subset $S \subset \Omega$. In many cases, one also wants to construct $F$ such that the \emph{accumulated cost}
\[
\sum_{k=0}^{K-1} c(x_k,F(x_k)), \quad K = \inf\{k\in\N\mid x_k\in T\},
\]
is minimized. 
In order to construct such a feedback, we employ the \emph{optimality principle}, cf.\ \cite{Be57a,Be87a},
\begin{equation}\label{eq:DPP}
V(x) = \inf_{u\in U}\{c(x,u)+V(f(x,u))\},\quad x\in\Omega \backslash T,
\end{equation}
where $V:\R^s\to [0,\infty]$ is the \emph{(optimal) value function}, with boundary conditions $V|_T=0$ and $V|_{\R^s\backslash\Omega}=\infty$.  Using $V$, we define a feedback by
\[
F(x)=\argmin_{u\in U}\{c(x,u)+V(f(x,u))\},
\]
whenever the minimum exists, e.g. if $V$ is continuous.

\subsubsection*{The Kru{\v{z}}kov transform.}

Typically, some part of the state space $\Omega$ will not be controllable to the target set $T$.  By definition, $V(x)=\infty$ for points $x$ in this part of $\Omega$.  An elegant way to deal (numerically) with the fact that $V$ might attain the value $\infty$, is by the (or more precisely one variant of the)  \emph{Kru{\v{z}}kov transform} $V\mapsto v=\exp(-V(\cdot))$, where we set $\exp(-\infty):= 0$, cf.\ \cite{Kr75a}.  After this transformation, the optimality principle (\ref{eq:DPP}) reads 
\[
v(x) = \sup_{u \in U} \left\{e^{-c(x,u)} v(f(x,u))\right\}, \quad x\in\Omega\backslash T,
\]
the boundary conditions transform to $v|_T=1$ and $v|_{\R^s\backslash\Omega}=0$.  The right hand side of this fixed point equation yields the \emph{Bellman operator}
\begin{equation}\label{eq:Bellman}
\Gamma(v)(x) := \left\{\begin{array}{ll}
\sup_{u \in U} \left\{e^{-c(x,u)} \bar v(f(x,u))\right\} & x\in\Omega\backslash T,\\
1 & x\in T,\\
0 & x\in\R^s\backslash\Omega
\end{array}\right.
\end{equation}
on the Banach space $L^\infty=L^\infty(\R^s,\R)$, where
\[
\bar v(x) := \left\{\begin{array}{ll}
v(x) & x\in\R^s\backslash T, \\
1 & x \in T.
\end{array}\right.
\]
By standard arguments and since we assumed $c$ to be bounded from below by $\delta > 0$ outside of the target set we obtain
$\|\Gamma(v)-\Gamma(w)\|_\infty \leq  L \|v-w\|_\infty$ with
\begin{equation}
\label{eq:L delta}
L = e^{- \delta} = \sup_{x\in\Omega\backslash T}\sup_{u\in U} e^{-c(x,u)} < 1,
\end{equation}
i.e.\ we have by the Banach fixed point theorem
\begin{lemma}
The Bellman operator $\Gamma:L^\infty\to L^\infty$ is a contraction and thus possesses a unique fixed point.
\end{lemma}

\section{Approximation with radial basis functions}\label{sec:RBF}

Only in simple or special cases (e.g.\ in case of a linear-quadratic problem), $V$ can be expressed in closed form. In general, we need to approximate it numerically.  Here, we are going to use \emph{radial basis functions} for this purpose, i.e.\ functions $\varphi_i:\R^s\to\R$ of the form 
$\varphi_i(x)=\varphi(\|x-x_i\|_2)$ on some set $X = \{x_1, \ldots, x_n\}\subset\R^s$ of \emph{nodes}. We assume the \emph{shape function} $\varphi:\R\to [0,\infty)$ to be nonnegative, typical examples include the Gaussian $\varphi(r) = \varphi^\sigma(r) = \exp(-(\sigma r)^2)$ and the \emph{Wendland functions} $\varphi (r) = \varphi^\sigma(r) = \max \{0, P(\sigma r)\}$, cf.\ \cite{We95a, Fa07a},
with $P$ an appropriate polynomial.  The \emph{shape parameter} $\sigma$ controls the ``width'' of the radial basis functions and has to be chosen rather carefully.  In case that the shape function $\varphi$ has compact support, we will need to require that the supports of the $\varphi_i$ cover $\Omega$.

\subsection{Interpolation.}

One way to use radial basis functions for approximation is by (``scattered data'') interpolation:  We make the ansatz   
\[
\tilde v (x)  = \sum_{i=1}^n c_i \varphi_i(x), \; \; c_i \in \R,
\]
for the approximate fixed point $\tilde v$ of (\ref{eq:Bellman}) and require $\tilde v$ to fulfill the interpolation conditions $\tilde v(x_i) = v_i $ for prescribed values $v_i,i=1,\ldots,n$. 
 The coefficient vector $c=(c_1,\ldots,c_n)$ is then given by  the solution of the linear system $Ac = v$ with $A = (\varphi_j(x_i))_{ij}$ and $v = (v_1,\ldots,v_n)$.

Formally, for some function $v:\Omega\to\R$, we can define its \emph{interpolation approximation} $Iv:\Omega\to\R$ by
\[
Iv = \sum_{i=1}^n v(x_i) u_i^*,
\]
where the $u_i^*:\Omega\to\R$ are \emph{cardinal basis functions} associated with the nodes $X$, i.e.\ a nodal basis with $u_i^*(x_i) = 1$ and $u_i^*(x_j) = 0$ for $i \neq j$. Note that $Iv$ depends linearly on $v$.

As we will see later (cf.\ Section~\ref{subsec:interp}), using interpolation for approximation has some shortcomings in our context.  As an improvement, we will use a least-squares type approach for function approximation known as \emph{Shepard's method}, cf.\ \cite{Fa07a}, which we will sketch in the following.

\subsection{Weighted least squares.}

Given some \emph{approximation space} $\cA = \text{span}(a_1$, $\ldots$, $a_m)$, $a_i:\Omega\to\R$, $m < n$, and a weight function $w:\Omega\to(0,\infty)$, we define the discrete inner product
\[
\langle f,g\rangle_w := \sum_{i=1}^n f(x_i)g(x_i)w(x_i),
\]
for $f,g:\Omega\to\R$ with  induced norm $\|\cdot\|_w$.  The \emph{weighted least squares approximant} $\tilde v\in\cA$ of some function $v:\Omega\to\R$ is then defined by minimizing $\|v-\tilde v\|_w$.  The solution is given by $\tilde v = \sum_{i=1}^m c_i a_i$, where the optimal coefficient vector $c=(c_1,\ldots,c_m)$ solves the linear system $Gc = v_\cA$ with Gram matrix $G=(\langle a_i,a_j\rangle_w)_{ij}$ and $v_\cA=(\langle v,a_j\rangle_w)_j$.

\subsection{Moving least squares.}

When constructing a least squares approximation to some function $f:\Omega\to\R$ at $x\in\Omega$, it is often natural to require that only the values $f(x_j)$ at some nodes $x_j$ close to $x$ should play a significant role.  This can be modeled by introducing a \emph{moving} weight function $w:\Omega\times\Omega\to\R$, where $w(\xi,x)$ is small if $\|\xi-x\|_2$ is large.  In the following, we will use a radial weight
\[
w(\xi,x)=\varphi(\|\xi-x\|_2),
\]
where $\varphi=\varphi^\sigma$ is the shape function introduced before. The corresponding discrete inner product is  
\[
\langle f,g\rangle_{w(\cdot,x)} := \sum_{i=1}^n f(x_i)g(x_i)w(x_i,x).
\]
The \emph{moving least squares approximation} $\tilde v$ of some function $v:\Omega\to\R$ is then given by $\tilde v(x) = \tilde v^x(x)$, where $\tilde v^x\in\cA$ is minimizing $\|v-\tilde v^x\|_{w(\cdot,x)}$. The optimal coefficient vector $c^x$ is given by the solution of the Gram system $G^xc^x=v^x_\cA$ with $G^x=(\langle a_i,a_j\rangle_{w(\cdot,x)})_{ij}$ and $v^x_\cA=(\langle v,a_j\rangle_{w(\cdot,x)})_j$.

\subsection{Shepard's method}

We now simply choose $\cA=\operatorspan (1)$ as approximation space. Then the Gram matrix is $G^x = \langle 1,1 \rangle_{w(\cdot,x)} = \sum_{i=1}^n w(x_i,x)$ and the right hand side is $v^x_\cA = \langle v, 1\rangle_{w(\cdot,x)} = \sum_{i=1}^n v(x_i)w(x_i,x)$. Thus we get $c^x = v^x_\cA/G^x = \sum_{i=1}^n v(x_i)\psi_i(x)$, where
\begin{equation}\label{eq:psi}
\psi_i(x) := \frac{w(x_i,x)}{\sum_{j=1}^n w(x_j,x)}.
\end{equation}
We define the \emph{Shepard approximation} $Sv:\Omega\to\R$ of $v:\Omega\to\R$ as
\[
Sv(x) = c^x\cdot 1 = \sum_{i=1}^n v(x_i)\psi_i(x), \quad x\in\Omega.
\]
Note again that $Sv$ depends linearly on $v$.  What is more, for each $x$, $Sv(x)$ is a convex combination of the values $v(x_1),\ldots,v(x_n)$, since
$\sum_{i=1}^n \psi_i(x) = 1$ for all $x\in\Omega$.

Shepard's method has several advantages over interpolation:  (a) Computing $Sv$ in a finite set of points only requires a matrix-vector product (in contrast to a linear solve for interpolation),  (b) the discretized Bellman operator remains a contraction, since the Shepard operator $S$ is non-expanding (cf.\ Lemma~\ref{lem:Shepard_non_expanding} in the next section), and (c) the approximation behavior for an increasing number of nodes is more favorable, as we will outline next.

\subsection{Stationary vs.\ non-stationary approximation}
 The number
\[
h := h_{X,\Omega} := \max_{x\in \Omega} \min_{\xi \in X} \| x-\xi\|
\]
is called the \emph{fill distance} of $X$ in $\Omega$, while
\[
q_X := \frac12 \min_{x,\xi\in X\atop x\neq \xi} \|x - \xi\|_2
\]
is the \emph{separation distance}. The fill distance is the radius of the largest ball inside $\Omega$ that is disjoint from $X$.  

In \emph{non-stationary} approximation with radial basis functions,  the shape parameter $\sigma$ is kept constant while the fill distance $h$ goes to $0$.  Non-stationary interpolation is convergent:
\begin{theorem}[\cite{Fa07a}, Theorem 15.3]
Assume that the Fourier transform $\hat\varphi$ of $\varphi$ fulfills 
\[
c_1(1 + \| \omega\|_2^2)^{-\tau} \leq \hat \varphi(\omega) \leq c_2 (1 + \| \omega\|_2^2)^{-\tau} 
\]
for some constants $c_1, c_2,\tau > 0$. In addition, let $k$ and $n$ be integers with $0 \leq n < k \leq \tau$ and $k > s/2$, and let $f \in C^k(\bar \Omega)$. Also suppose that $X = \{ x_1, \dots, x_n \} \subset \Omega$ satisfies $\diam(X) \leq 1$ with sufficiently small fill distance. Then for any $1 \leq q \leq \infty$ we have
\[
|f - If|_{W_q^n(\Omega)} \leq c\rho_X^{\tau-x} h^{k-n-s(1/2-1/q)_+} \|f\|_{C^k(\bar \Omega)},
\]
where $\rho_X = h/q_X$ is the so-called mesh ratio for $X$.
\end{theorem}
On the other hand, \emph{stationary} interpolation, i.e.\ letting $\sigma$ shrink to $0$ along with $h$ does not converge (for a counter example see \cite{Fa07a}, Example 15.10). 

For Shepard's method instead of interpolation, though, the exact opposite holds: Non-stationary approximation does not converge (cf.\ \cite{Fa07a}, Ch.\ 24), while stationary approximation does, as we will recall in Lemma \ref{shepard stationary convergence} below.  In practice, this is an advantage, since we can keep the associated matrices sparse. 

\section{Discretization of the optimality principle}
\label{chapRBFshepard} 

We now want to compute an approximation to the fixed point of the Bellman operator (\ref{eq:Bellman}) by \emph{value iteration}, i.e.\ by iterating $\Gamma$ on some initial function $v^{(0)}$.  We would like to perform this iteration inside some finite dimensional approximation space $\cW\subset L^\infty$, i.e.\ after each application of $\Gamma$, we need to map back into $\cW$. 

\subsection*{Interpolation}\label{subsec:interp}

Choosing $\cW=\operatorspan(\varphi_1,\ldots,\varphi_n)$, we define the \emph{Bellman interpolation operator} to be
\[
\hat \Gamma := I\circ \Gamma:\cW\to\cW.
\]
In general, the operator $I$ is expansive and not necessarily monotone.  For that reason, one cannot rely on the iteration with $\hat \Gamma$ to be convergent (although in most of our numerical experiments it turned out to converge) and move our focus towards the value iteration with Shepard's method.

\subsection*{Shepard's method.}

With $\cW=\operatorspan(\psi_1,\ldots,\psi_n)$ (cf.\ (\ref{eq:psi})), we define the \emph{Bellman-Shepard operator} as
\[
\tilde \Gamma := S\circ \Gamma:\cW\to\cW.
\]
Explicitly, the value iteration reads
\begin{equation}\label{eq:value iteration}
v^{(k+1)} := S\left(\Gamma[v^{(k)}]\right), \quad k=0,1,2,\ldots,
\end{equation}
where, as mentioned, some initial function $v^{(0)}\in\cW$ has to be provided.

\begin{lemma}\label{lem:Shepard_non_expanding}
The Shepard operator $S:(L^\infty,\|\cdot\|_\infty)\to (\cW,\|\cdot\|_\infty)$ has norm~$1$.
\end{lemma}

\begin{proof}
Since, by assumption, the $\psi_i$ are nonnegative and, as mentioned above, for each $x\in\Omega$, $Sv(x)$ is a convex combination of the values $v(x_1),\ldots,v(x_n)$, we have for each $x\in\Omega$
\begin{align*}
|Sv(x)| \leq \sum_{i=1}^n |v(x_i)\psi_i(x)| \leq \max_{i=1,\ldots,n}|v(x_i)| \sum_{i=1}^n |\psi_i(x)| \leq \max_{i=1,\ldots,n}|v(x_i)| \leq \|v\|_\infty,
\end{align*}
so that $\|Sv\|_\infty\leq \|v\|_\infty$.  Moreover, for constant $v$ one has $\|Sv\|_\infty= \|v\|_\infty$.
\end{proof}

As a composition of the contraction $\Gamma$ and the Shepard operator $S$, we get
\begin{theorem}\label{thm:value iteration}
The Bellman-Shepard operator $\tilde\Gamma:(\cW,\|\cdot\|_\infty)\to (\cW,\|\cdot\|_\infty)$ is a contraction, thus the iteration (\ref{eq:value iteration}) converges to the unique fixed point $\tilde v\in\cW$ of $\tilde\Gamma$.
\end{theorem}

\subsection*{Convergence for decreasing fill distance}




\label{subsec:ConvergenceOfShepardApproximants}

Assume we are given a Lipschitz function $v:\Omega\to\R$ and a sequence of sets of nodes $X_k$ with fill distances $h_k$.
We consider the Shepard operators $S_k$ associated with the sets $X_k$ and shape parameters $\sigma_k := C_\sigma/h_k$ for some constant $C_\sigma>0$. 
Under the assumption that $h_k \rightarrow 0$ we get convergence of the respective Shepard approximations $S_kv$ to $v$ as shown in the following result, which is a minor generalization of the statement preceding theorem 25.1 in \cite{Fa07a} from $C^1$ functions to Lipschitz functions.
Here, we consider only shape functions with compact support.

\begin{lemma}
\label{shepard stationary convergence}
Let $\Omega \subset \R^s$ and $v: \Omega \rightarrow \R$ be Lipschitz continuous with Lip\-schitz constant $L_v$.
Then 
\[
\|v - S_kv\|_\infty \leq  L_v \frac{\rho}{C_\sigma} h_k
\]
where $\rho$ is a number such that $U_{\rho}(0) \supset \supp(\varphi)$ and $\varphi = \varphi^1$ is the shape function with shape parameter $\sigma = 1$.
\end{lemma}

\begin{proof}
By the scaling effect of $\sigma$ we have
\[
U_{\rho / \sigma}(0) \supset \supp(\varphi^\sigma)
\]
as well as
\[
\sup_{x \in U_{\rho / \sigma} (x_0)} |v(x)-v(x_0)| \leq L_v \frac{\rho}\sigma, \;\;\; x_0 \in \Omega
\]
and, as a consequence,
\[
|v(x) - Sv(x)| \leq L_v \frac{\rho}\sigma = L_v \frac{\rho}{C_\sigma} h_k 
\]
because the Shepard approximation is given by a convex combination of values of $v$ inside a $\frac{\rho}{\sigma}$-neighborhood.
\end{proof}

Note that the value functions $V$ resp.\ $v$ considered above are Lipschitz continuous (a proof is given in the Appendix).  

We now show that this Lemma implies convergence of the approximate value functions for decreasing fill distance.

\begin{theorem}
\label{ConvergenceOfShepardApproximants}
Assume that $f:\Omega\times U\to\Omega$  in (\ref{eq:control_system}) is continuously differentiable and that the associated cost function $c:\Omega\times U\to [0,\infty)$ is Lipschitz continuous in $x$.
Let a sequence of finite sets of nodes $X_k\subset \Omega $ with fill distances $h_k>0$ be given and let $\sigma_k = C_\sigma / h_k$, $C_\sigma>0$, be the associated sequence of shape parameters as well as  $S_k$ the associated Shepard approximation operators.  Let $v$ be the fixed point of $\Gamma$ and $\tilde v_k$ the fixed points of $\tilde \Gamma_k = S_k \circ \Gamma$. Then
\[
\| v-\tilde v_k \|_\infty 
\leq \frac{L_v \rho}{C_\sigma (1-e^{-\delta})} h_k.
\]
\end{theorem}
\begin{proof}
Let $e_k$ be the norm of the residual of $v$ in the Bellman-Shepard equation, i.e.
\[
e_k = \| v-\tilde\Gamma_k(v) \|_\infty = \| v-S_kv\|_\infty.
\]
Then 
\begin{align*}
\| v-\tilde v_k \|_\infty 
& \leq  \| v-\tilde\Gamma_k(v) \|_\infty + \|\tilde v_k - \tilde\Gamma_k (v)\|_\infty \\
& = e_k + \|\tilde\Gamma_k (\tilde v_k) - \tilde\Gamma_k (v) \|_\infty \\
& \leq e_k + e^{-\delta} \|\tilde v_k - v \|_\infty.
\end{align*}
Consequently,
\[
\| v-\tilde v_k \|_\infty 
\leq \frac {e_k} {1- e^{-\delta}}
= \frac{L_v \rho}{C_\sigma (1-e^{-\delta})} h_k   .
\]
\end{proof}

\subsection*{Construction of a stabilizing feedback}

As common in dynamic programming, we now use the approximate value function $\tilde V(x)=-\log(\tilde v(x))$, $x \in S:=\{x\in \Omega: \tilde V(x) < \infty\}$, in order to construct a feedback which stabilizes the closed loop system on a certain subset of $\Omega$. This feedback is
\[
\tilde u(x) := \argmin_{u \in U} \{ c(x,u)+\tilde V(f(x,u)) \}, \quad x\in S.
\]
Note that the argmin exists, since $U$ is compact and $c$ and $\tilde V$ are continuous.  
We define the \emph{Bellman residual}
\begin{align*}
e(x) &:= \inf_{u \in U} \{ c(x,u) + \tilde V(f(x,u)) \} - \tilde V(x), \quad x\in S.
\end{align*}
and show that $\tilde V$ decreases along a trajectory of the closed loop system if the set
\[
R_\eta := \{x\in S \mid e(x) \leq \eta \tilde c(x)\}, \quad \eta\in (0,1),
\]
where the Bellman residual is (at least at a constant factor) smaller than $\tilde c(x) := c(x,\tilde u(x))$, contains a sublevel set of $\tilde V$ and we start in this set.
\begin{proposition}\label{prop:decay}
Suppose that $D_C = \{x\in S\mid \tilde V(x) < C\}\subset R_\eta$ for some $C > 0$. Then for any $x_0\in D_C$, the associated trajectory generated by the \emph{closed loop system} 
$x_{j+1} = f(x_j,\tilde u(x_j))$, $j=0,1,\ldots$, stays in $D_C$ and satisfies 
\[
\tilde V(x_\ell) \leq \tilde V(x_0) - (1-\eta)\sum_{j=0}^{\ell-1} c(x_j,\tilde u(x_j)).
\]
\end{proposition}
\begin{proof}
Since $e(x) = \tilde c(x) + \tilde V(f(x,\tilde u(x))) - \tilde V(x)$, we have for $x_j\in D_C\subset R_\eta$
\begin{align*}
\tilde V(x_{j+1}) & =  \tilde V(x_j) - \tilde c(x_j) + e(x_j) < \tilde V(x_j) < C
\end{align*}
thus $x_{j+1}\in D_C$, which shows that the closed loop trajectory stays in $D_C$. Further, 
\begin{align*}
\tilde V(x_{j+1}) & =  \tilde V(x_j) - c(x_j,\tilde u(x_j)) + e(x_j) \\
& \leq \tilde V(x_j) - (1 - \eta) c(x_j,\tilde u(x_j)).
\end{align*}
which shows the decay property.
\end{proof}

%
%

\section{Implementation and numerical experiments}
\label{sec:experiments}

\subsection*{Implementation}

A function $v\in\cW=\operatorspan(\psi_1,\ldots,\psi_n)$ is defined by the vector $\hat v=(\hat v_1,\ldots,\hat v_n)\in\R^n$ of its coefficients, i.e.\ $v = \sum_{i=1}^n \hat v_i \psi_i$. We can evaluate $v$ on an arbitrary set of points $Y=\{y_1,\ldots,y_e\}\subset \Omega$ by the matrix-vector product $A(Y)\hat v$, where $A(Y)$ is the $e\times n$-matrix with entries
$a_{ij} = \psi_j(y_i)$.

In order to compute $\hat v^{(k+1)}$ in the value iteration (\ref{eq:value iteration}), we need to evaluate $\Gamma[v^{(k)}]$ on $X$, i.e.\ we have to compute 
\[
\Gamma[v^{(k)}](x_i)=\left\{\begin{array}{ll}
\sup_{u \in U} \left\{e^{-c(x_i,u)} (\bar v^{(k)}(f(x_i,u)))\right\} & x_i\in X\backslash T,\\
1 & x_i\in X\cap T.
\end{array}\right.
\]
In general, this is a nonlinear optimization problem for each $x_i$.  For simplicity, we here choose to solve this by simple enumeration, i.e.\ we choose a finite set $\tilde U=\{u_1,\ldots,u_m\}$ of control values in $U$ and approximate 
\begin{equation}\label{eq:rhs}
\Gamma[v^{(k)}](x_i)\approx \max_{j=1,\ldots,m} \left\{e^{-c(x_i,u_j)} (\bar v^{(k)}(f(x_i,u_j)))\right\}
\end{equation}
for each $x_i\in X\backslash T$.  Let $Y'=\{f(x_i,u_j)\mid i=1,\ldots,n,j=1,\ldots,m\}$, then the values $v^{(k)}(f(x_i,u_j))$ are given by the matrix-vector product $A(Y')\hat v^{(k)}$. From this, the right hand side of (\ref{eq:rhs}) can readily be computed.  The \textsc{Matlab} codes for the following numerical experiments can be downloaded from the homepages of the authors.

In some cases, our numerical examples are given by restrictions of problems on $\R^s$ to a compact domain $\Omega \in \R^s$. In general the dynamical system on $\R^s$ is a map $f_1: \R^s \times U \to \R^s$ which does not restrict to a map $f_2 := f_1|_{\Omega}: \Omega \times U \to \Omega$. This can be achieved by replacing $f_2$ with $f := \Pi \circ f_2$ where $\Pi$ is a Lipschitz-continuous projection of $\R^s$ into $\Omega$. In our numerical examples it turned out that there is no visible difference between using or not using the projection into $\Omega$. Consequently, we do not use it in our published matlab code in order to keep things simpler.

\subsection{A simple 1D example}

We begin with the simple one-dimensional system
\[
f(x,u) = a u x
\]
on $\Omega = [0,1], U = [-1, 1]$, with parameter $a = 0.8$ and cost function
$
c(x,u) = ax.
$
Apparently, the optimal feedback is $u(x) = -1$, yielding  the optimal value function $V(x) = x$.  We choose equidistant sets $X_k=\{0,1/k,\ldots,1-1/k,1\}$ of nodes, $\tilde U=\{-1,-0.9,\ldots,0.9,1\}$, $T=[0,1/(2k)]$ and use the Wendland function $\phi^\sigma (r) = \max\{ 0, (1-\sigma r)^4 (4\sigma r+1) \}$ as shape function with parameter $\sigma = k/5$. In Figure~\ref{fig:1Derror}, we show the $L^\infty$-error of the approximate value function $\tilde V_k$ in dependence on the fill distance $h_k=1/k$ of the set of nodes $X_k$.  As predicted by Theorem \ref{ConvergenceOfShepardApproximants}, we observe a linear decay of the error in $h_k$. 

\begin{figure}[h]
  \centering
  \includegraphics[width=0.6\textwidth]{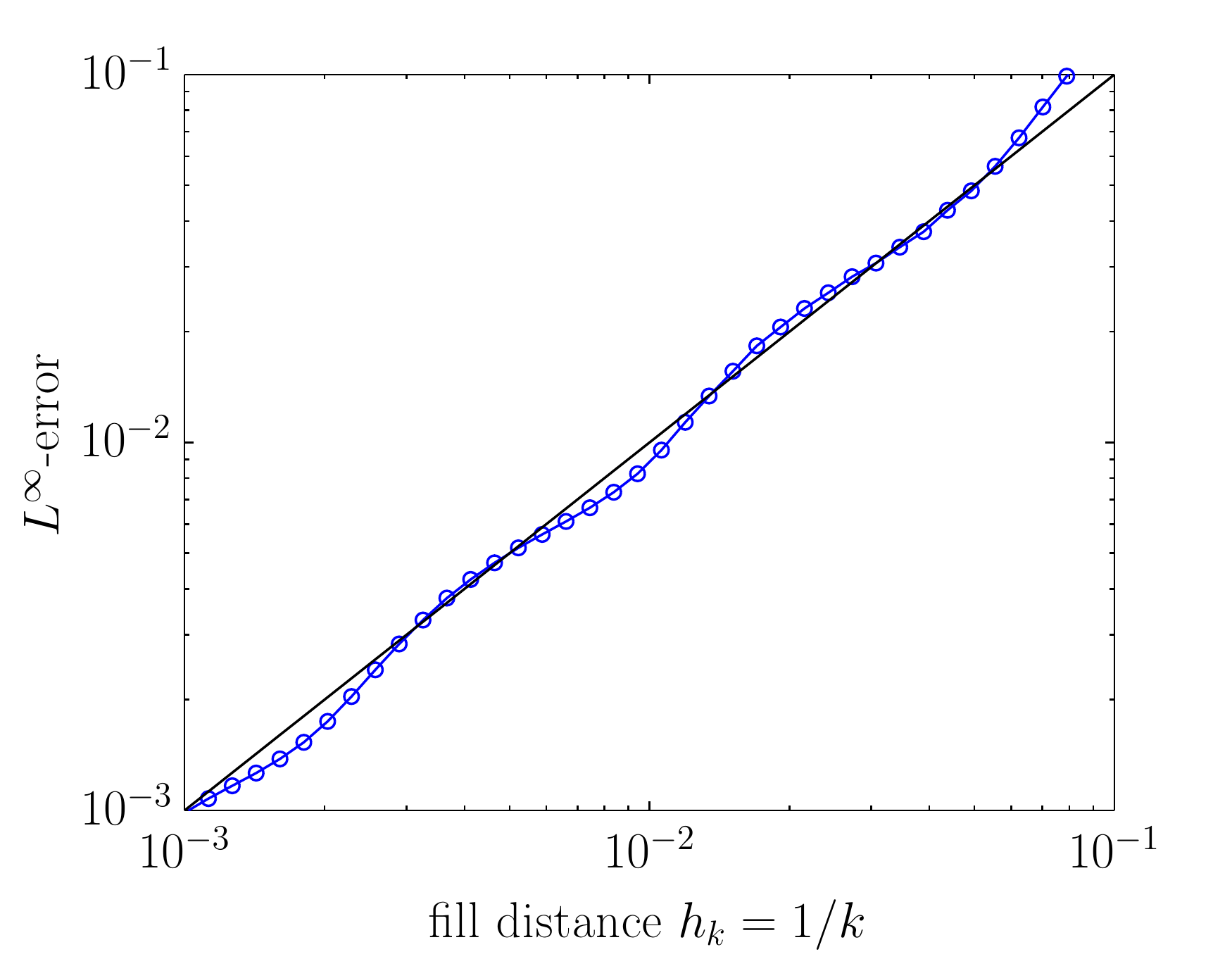}
  \caption{
$L^\infty$-error of the approximate value function $V_k=-\log(\tilde v_k)$ in dependence on the fill distance $1/k$, where $k$ is the number of nodes.}
  \label{fig:1Derror}
\end{figure}

\subsection{Example: shortest path with obstacles}

Our next example is supposed to demonstrate that state space constraints can trivially be dealt with, even if they are very irregular: We consider a boat in the mediteranian sea surrounding Greece (cf.~Fig.~\ref{fig:short}) which moves with constant speed $1$.  The boat is supposed to reach the harbour of Athens in shortest time. Accordingly, the dynamics is simply given by
\[
f(x,u) = x + hu,
\]
where we choose the time step $h=0.1$, with $U := \{ u \in \R^2 : \|u\| = 1 \}$, and the associated cost function by
$
c(x,u) \equiv 1.
$
In other words, we are solving a shortest path problem on a domain with obstacles with complicated shape. 

In order to solve this problem by our approach, we choose the set of nodes $X$ as those nodes of an equidistant grid which are placed in the mediteranean sea within  the rectangle shown in Fig.\ (\ref{fig:short}), which we normalize to $[-10,10]^2$.  We extracted this region from a pixmap of this region with resolution 275 by 257 pixels.  The resulting set $X$ consisted of $50301$ nodes. We choose $U=\{\exp(2\pi i j/20):j=0,\ldots,19\}$.  The position of Athens in our case is approximately given by $A=(-4, 4)$ and we choose $T=A+0.004\cdot[-1,1]^2$. We again use the Wendland function $\phi^\sigma (r) = \max\{ 0, (1-\sigma r)^4 (4\sigma r+1) \}$ as shape function with parameter $\sigma = 10$.

In Figure~(\ref{fig:short}), we show some isolines of the approximate optimal value function.  The computation took around 10 seconds on our $2.6$ GHz Intel Core i5.

\begin{figure}[h]
  \centering
  \includegraphics[width=0.66\textwidth]{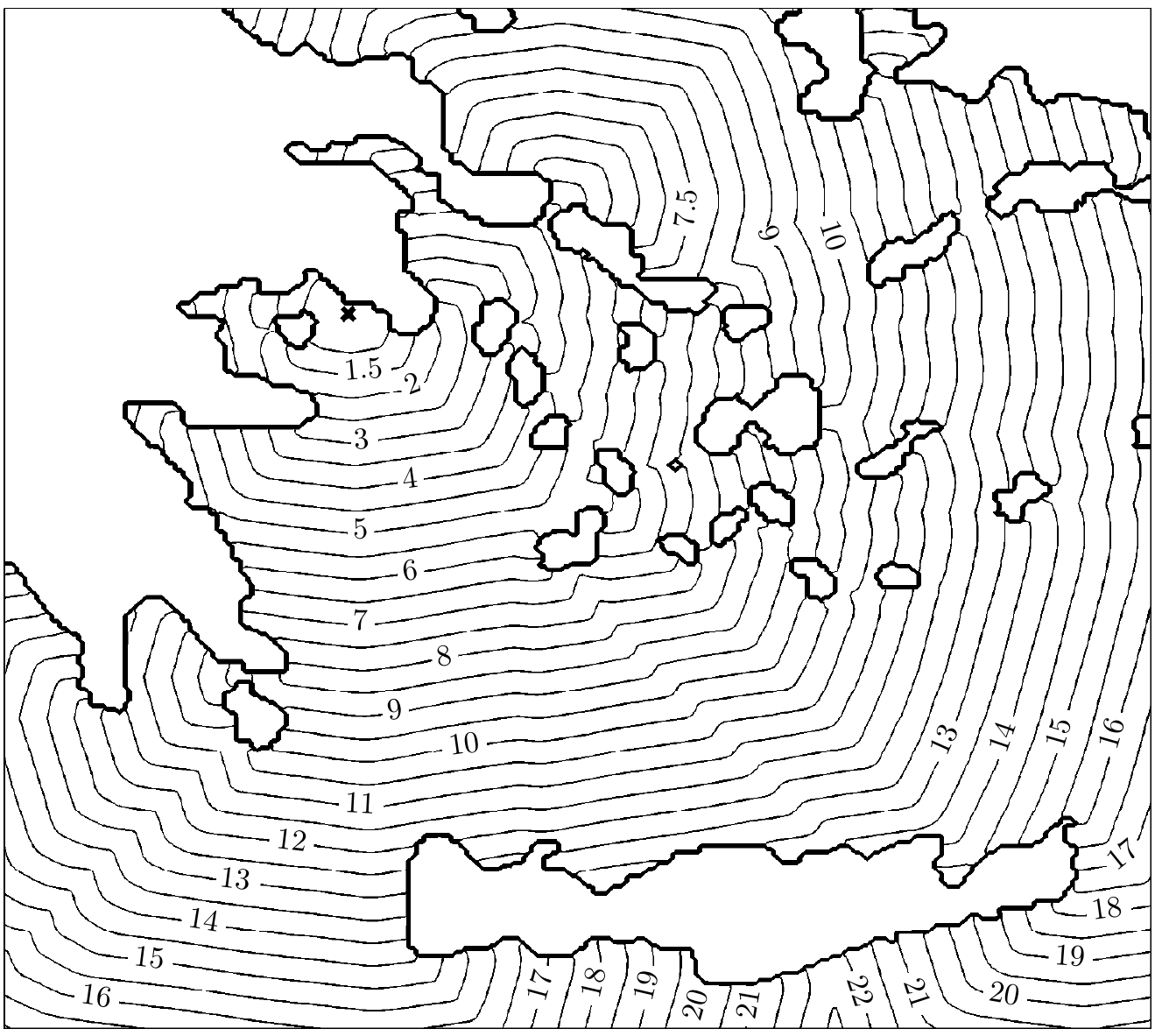}
  \caption{
Isolines of the optimal value function for the shortest path example.}
  \label{fig:short}
\end{figure}

\subsection{Example: an inverted pendulum on a cart}
\label{sec:pendulum}

Our next example is two-dimensional as well, with only box constraints on the states, but the stabilization task is more challenging:  We consider balancing a planar inverted pendulum on a cart that moves under an applied horizontal force, cf.\ \cite{JuOs04a} and Figure~\ref{fig:pendulum}.

\begin{figure}[htb]
  \centering
  \includegraphics[width=0.3\textwidth]{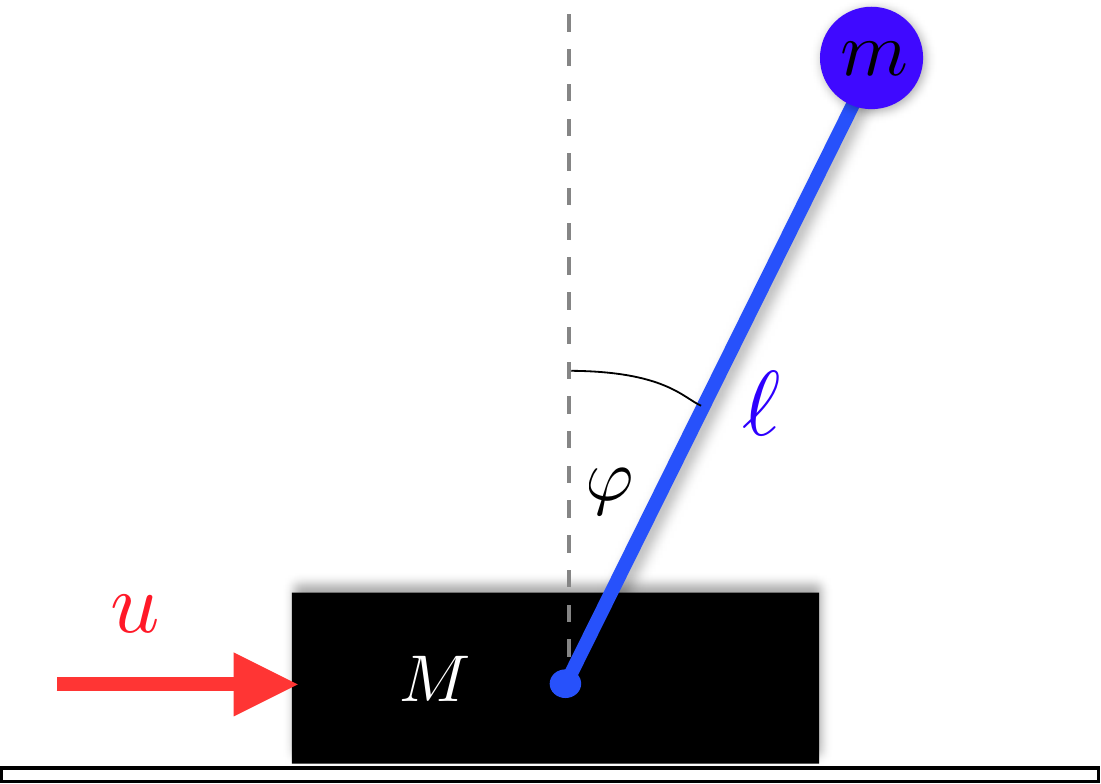}
  \caption{Model of the inverted pendulum on a cart.}
  \label{fig:pendulum}
\end{figure}

The configuration of the pendulum is given by the offset angle $\varphi$ from the vertical up position and we do not consider the motion of the cart.  Correspondingly, the state of the system is $x=(x_1,x_2):=(\varphi, \dot{\varphi}) \in \R^2$.  The equation of motion becomes
\begin{equation}
\label{eq:pend}
m_r \, \cos^2 (\varphi)\,\ddot{\varphi}  - \frac{1}{2} \, m_r \, \sin (2 \varphi)\,\dot\varphi^2   + \frac{g}{\ell} \, \sin (\varphi) - \frac{4}{3} - \frac{m_r}{m \,\ell} \, \cos (\varphi) \, u = 0,
\end{equation}
where $M = 8\,\mbox{kg}$ is the mass of the cart, $m = 2\,\mbox{kg}$ the
 mass of the pendulum and $\ell = 0.5\,\mbox{m}$ is the distance of the center of mass from the pivot. We use $m_r = m / (m + M)$ for the mass ratio and $g = 9.8
\mbox{m}/\mbox{s}^2$ for the gravitational constant.  The stabilization
of the pendulum is subject to the cost
$c(x,u) = c((\varphi,\dot\varphi),u) = \frac{1}{2} (0.1 \varphi^2 + 0.05 \dot\varphi^2 + 0.01 u^2)$.
For our computations, we need to obtain a discrete-time control system. To this end, we consider the time sampled system with sampling period $h$ and keep the control $u(t)$ constant during the sampling period.  In the computations, we used   $h = 0.1$ and computed the time sampling map via one explicit Euler step. 
As in \cite{JuOs04a}, we choose $\Omega = [-8, 8] \times [-10,10]$ as the region of interest and the set $X$ of nodes by a tensor product grid, using $100$ equally spaced points in each coordinate direction (cf.\ the Matlab code in the next section), and $\tilde U=\{-128,120,\ldots,120,128\}$.  We again use the Wendland function $\phi^\sigma (r) = \max\{ 0, (1-\sigma r)^4 (4\sigma r+1) \}$ as shape function, the shape parameter $\sigma$ is chosen such that the support of each $\varphi_i$ overlaps with the supports of roughly $20$ other $\varphi_i$'s, i.e.\ $\sigma\approx 2.22$ here.

\begin{figure}[H]
  \centering
  \includegraphics[width=0.55\textwidth]{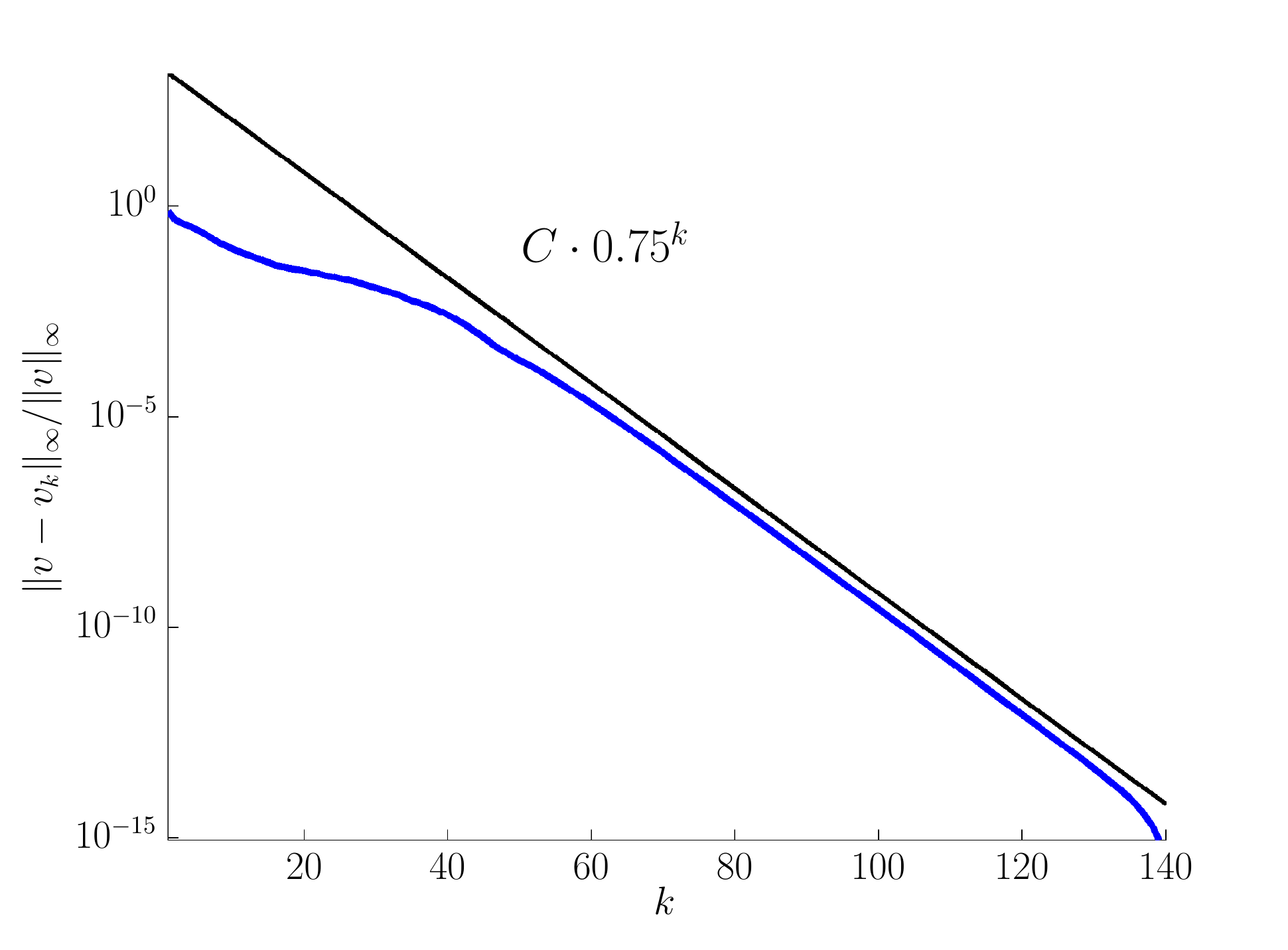}
  \caption{Relative $L^\infty$-error of $\tilde v_k$ in dependence on the number $k$ of iterations in the fixed point iteration (\ref{eq:value iteration}).  Here, we used the iterate $v_{149}$ as an approximation to the true fixed point $v$.}
  \label{fig:error}
\end{figure}

In Figure \ref{fig:error}, we show the behavior of the (relative) $L^\infty$-approximation error during the value iteration. We observe geometric convergence as expected by Theorem~\ref{thm:value iteration}. The computation of the optimal value function took 6.8 seconds. 

 In Figure~\ref{fig:value}, some isolines of the approximate value function are shown, together with the complement of the set $R_1$ (cf.\ Proposition~\ref{prop:decay}) as well as the first 80 points of a trajectory of the closed loop system with the optimal feedback.  At the 30th iteration, the trajectory leaves the set $R_1$ and subsequently moves away from the target.

\begin{figure}[htb]
  \centering
  \includegraphics[width=0.9\textwidth]{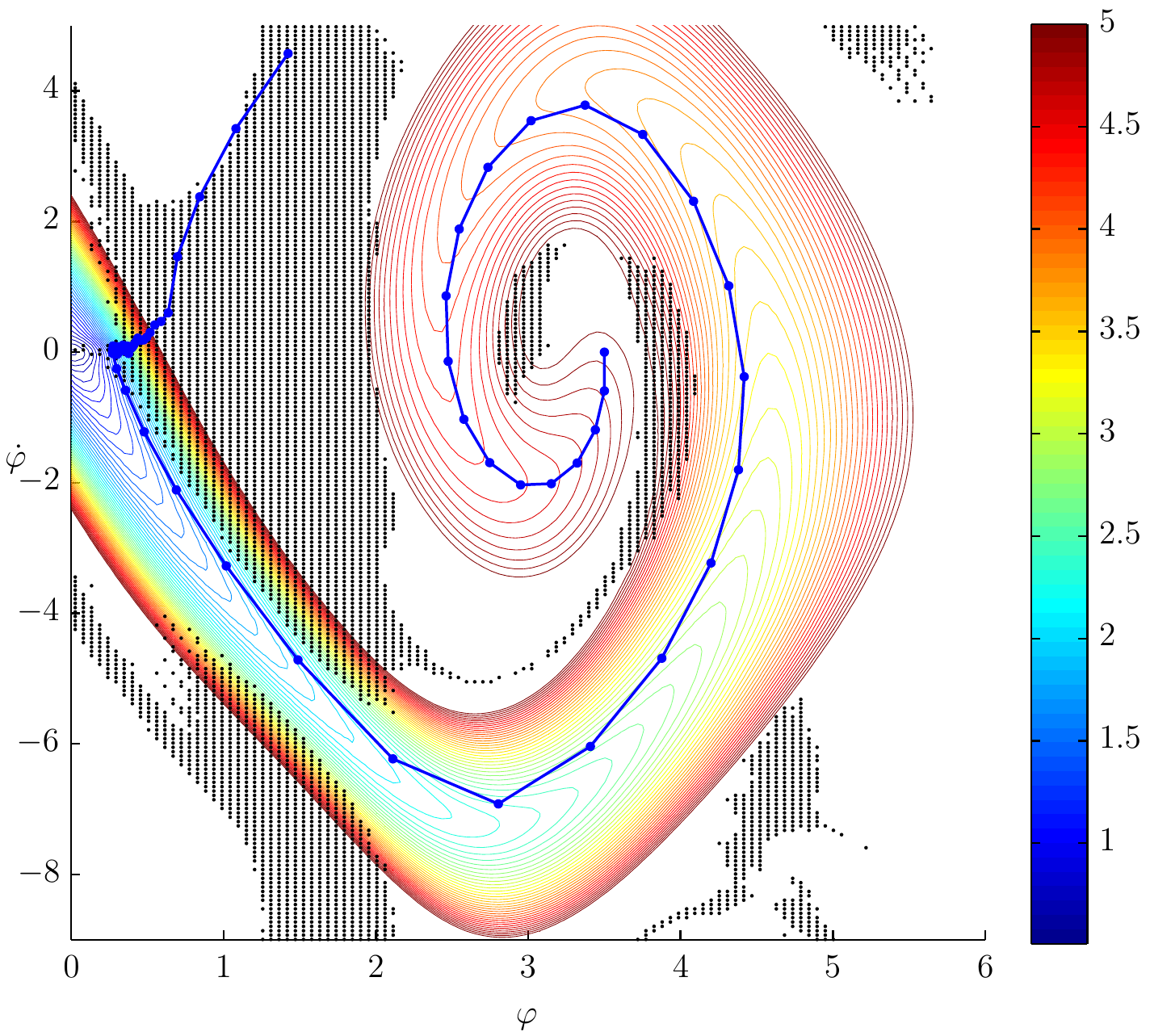}
  \caption{Approximate optimal value function $\tilde V=-\log(\tilde v(\cdot))$ (isolines color coded) for the planar inverted pendulum, together with the set $\Omega\backslash R_1$, where the Bellman residual $e$ is larger than $\tilde c$ (black dots) and the trajectory of the closed loop system starting at the initial value $(3.5,0)$ (blue). Close to the target, this trajectory leaves $R_1$ and subsequently moves away from target.}
  \label{fig:value}
\end{figure}

In Figure~\ref{fig:error_meshsize}, the behavior of the $L^\infty$ error of the approximate optimal value function in dependence of the fill distance $h$ is shown.  Here, we used the value function with fill distance $h=0.02$ as an approximation to the true one. Again, the convergence behavior is consistent with Theorem~\ref{ConvergenceOfShepardApproximants} which predicts a linear decay of the error. The corresponding Matlab code is given in Code~1.

\begin{figure}[t]
  \centering
  \includegraphics[width=0.6\textwidth]{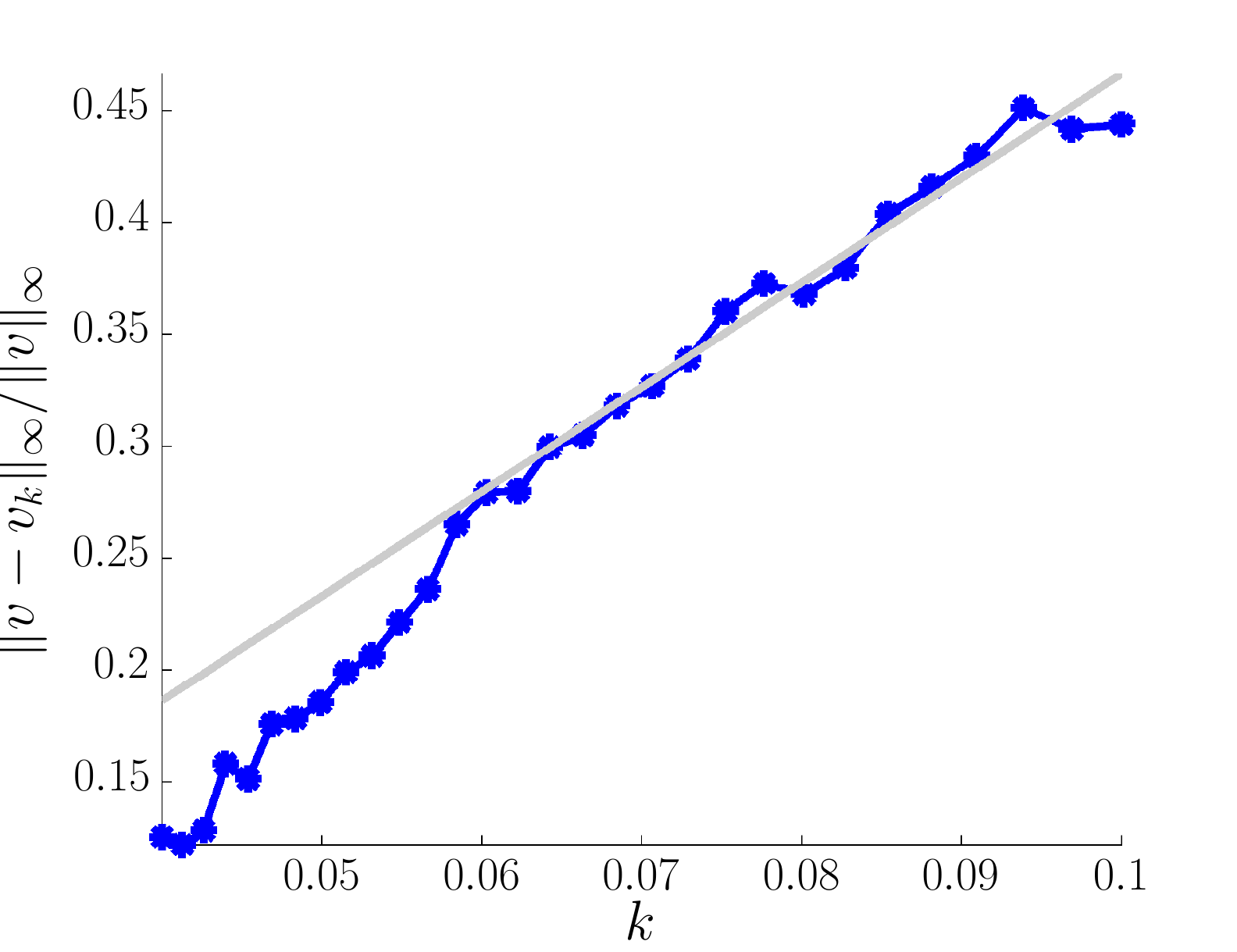}
  \caption{Relative $L^\infty$-error of the approximate optimal value function $v_k$ for the inverted pendulum in dependence on the fill distance of the nodes. Here, we used the value function for fill distance $0.02$ as an approximation to the true one.}
  \label{fig:error_meshsize}
\end{figure}

\begin{Code}[H]
\begin{lstlisting}[breaklines=true,linerange={4-26}]
%% an inverted pendulum on a cart
M = 8; m = 2; mr = m/(m+M); l = 0.5; c = 9.8; h = 0.1; tic
f = @(x,u) kron(x,ones(length(u),1))+h*[reshape(ones(length(u),1)*x(:,2)', length(u)*size(x,1),1) , reshape((c/l*ones(length(u),1)*sin(x(:,1)') -1/2*mr*ones(length(u),1)*(x(:,2)'.^2.*sin(2*x(:,1)')) -mr/m/l*u*cos(x(:,1)'))./(4/3-mr*ones(length(u),1)*cos(x(:,1)').^2), length(u)*size(x,1),1)];   
c = @(x,u) exp(-h*(1/2*reshape(0.1*ones(length(u),1)*x(:,1)'.^2 + ...
    0.05*ones(length(u),1)*x(:,2)'.^2 + 0.01*u.^2*ones(1,size(x,1)), ...
    length(u)*size(x,1),1)));
phi = @(r) max(spones(r)-r,0).^4.*(4*r+spones(r));     % Wendland function 
T = [0 0]; v_T = 1;                                    % boundary cond. 
shepard = @(A) spdiags(1./sum(A')',0,size(A,1),size(A,1))*A; % Shepard op.
S = [8,10];                                            % radius of domain
L = 33; U = linspace(-128,128,L)';                     % control values
N = 100; X1 = linspace(-1,1,N); 
[XX,YY] = meshgrid(X1*S(1),X1*S(2)); X = [XX(:) YY(:)];% nodes
ep = 1/sqrt((4*prod(S)*20/N^2)/pi);                    % shape parameter
A = shepard(phi(ep*sdistm(f(X,U),[T;X],1/ep)));        % Shepard matrix
C = c(X,U);                                            % one step costs
%% value function 
v = zeros(N^2+1,1); v0 = ones(size(v)); TOL = 1e-12;   % value iteration      
while norm(v-v0,inf)/norm(v,inf) > TOL                            
    v0 = v; 
    v = [v_T; max(reshape(C.*(A*v),L,N^2))'];          % Bellman operator
end
contour(reshape(min(-log(abs(v(2:end))),15),N,N),linspace(0,5,60));
\end{lstlisting}
\caption{Matlab code for the inverted pendulum example. Here, \texttt{A = sdistm(X,Y,r)} is the sparse matrix of pairwise euclidean distances between the points in the rows of $X$ and $Y$ not exceeding distance $r$. An implementation can be downloaded from the homepages of the authors.
}
\end{Code}

\newpage
\subsection{Example: magnetic wheel}

We finally turn to an example with a three dimensional state space, the stabilization  of a \emph{magnetic wheel}, used in magnetic levitation trains, cf. \cite{GoMeMu79} and Figure \ref{fig:wheelfigure}.

\begin{figure}[H]
  \includegraphics[width=0.45\textwidth]{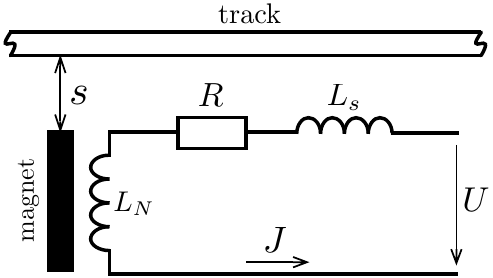}
\caption{Model of the magnetic wheel.}
\label{fig:wheelfigure}
\end{figure}

A point in state space is given by the gap
$s$ between the magnet and the track, its change rate $\dot s = v$ and the electrical current $J$ through the magnet.  The input is the voltage $U$ applied to the circuit. The dynamics is given by
\begin{align*}
\dot s &= v , \\
 \dot v &= \frac {CJ^2} {m_m4s^2}  -  \mu g , \\
\dot J &=  \frac1{L_s + \frac C{2s} } (- R J + \frac {C} {2s^2} Jv +   U) ,
\end{align*}
where $ C = L_N2s_0 $, the target gap $s_0 = 0.01$, the inductance $L_N = 1$ of the magnet, the magnet mass $m_m = 500$, the ratio of the total mass and the magnet mass $\mu = 3$, the resistance $R = 4$, the leakage inductance $L_s = 0.15$ and the 
gravitational constant $g = 9.81$.  We consider the cost function
\begin{equation}\label{eq:costcontM}
        c(x,u) = \frac{1}{2} (100 s^2 + v^2 + 0.002 u^2).
\end{equation}
The model has an unstable equilibrium at approximately $x_0:=(s_0,v_0,J_0) =(0.01$, $0, 17.155)$ which, again, we would like to stabilize by an optimal feedback.  We choose $\Omega = [0,0.02]\times [-4,4] \times [J_0-80,J_0+80]$ as state space, $\tilde U=\{6\cdot 10^3u^3\mid u\in \{-1,-0.99,\ldots,0.99,1\}\}$ as the set of controls, an equidistant grid $X$ of $30\times 30\times 30$ nodes in $\Omega$, the Wendland function $\phi^\sigma (r) = \max\{ 0, (1-\sigma r)^4 (4\sigma r+1) \}$ as shape function with shape parameter $\sigma=11.2$, such that the support of each $\varphi_i$ overlaps with the supports of roughly $10$ other $\varphi_i$'s.  The computation of the value function takes around 60 seconds.  In figure \ref{fig:magnet} we show a subset of the stabilizable subset of $\Omega$, i.e.\ we show the set $\{x\in \Omega\mid \tilde v(x) > 10^{-20}\}$.

\begin{figure}
  \centering
  \includegraphics[width=0.49\textwidth]{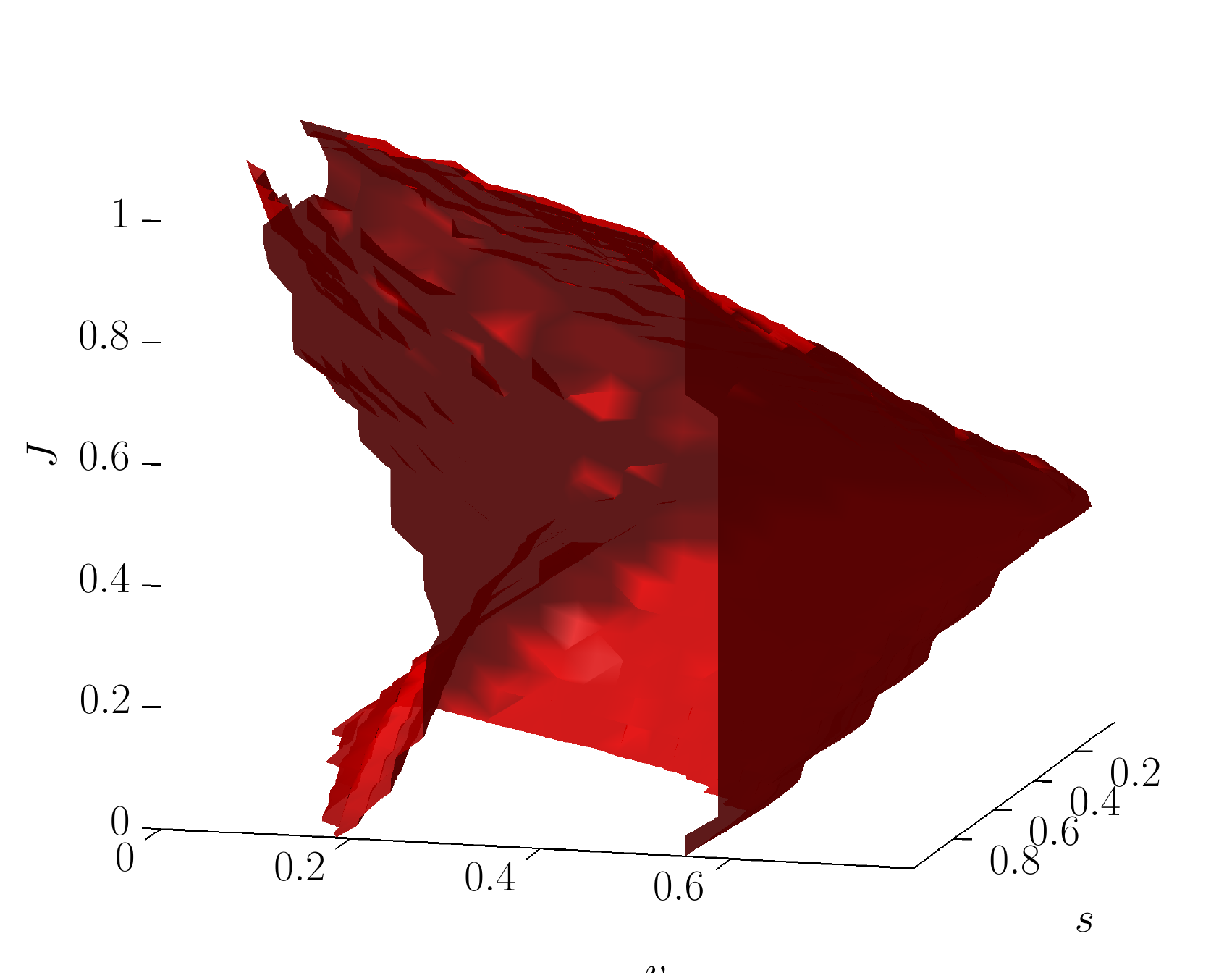} 
 \includegraphics[width=0.49\textwidth]{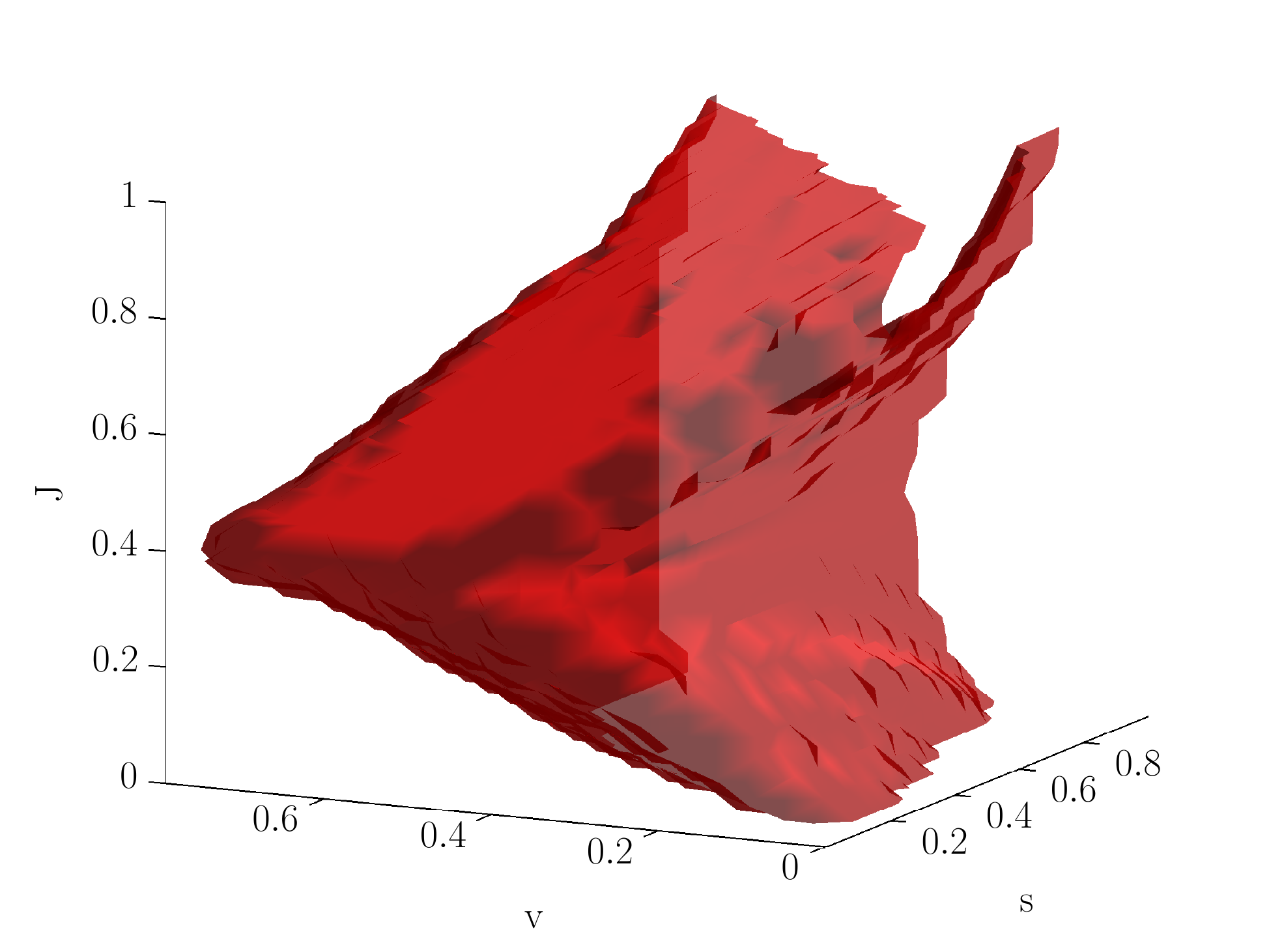}
\caption{Magnetic wheel example: The set $\{x\in \Omega\mid \tilde v(x) > 10^{-20}\}$ from two perspectives.}
\label{fig:magnet}
\end{figure}

\section{Future directions}

There are many ways in which the approach from this paper can be extended.  For example, it is natural not to fix the shape parameter $\sigma$ at the same value for all basis functions $\varphi_i$.  One could, e.g., either try to choose $\sigma$ (greedily) in an optimal way for each $\varphi_i$ or to implement a multilevel type scheme which works with a scale of values for $\sigma$.  This also raises the issue of an improved or even optimal choice for the set $X$ of nodes (instead of the equidistant grid used here).  One drawback of the value iteration used here is that the images of all possible node/control value pairs have to be computed and the corresponding matrix entries have to be stored.  It would be nice to have some sort of ``fast marching'' type algorithm which needs this reachability information only locally. Finally, the approximation space that we use here is rather smooth -- in contrast to the value function which in general is only Lipschitz-
continuous. If one considers a \emph{relaxed} version of the optimality principle, e.g.\ in the sense of \cite{LiRa06a}, smoother solutions might exist which can be approximated with higher efficiency.

\begin{appendix}

\section{Lipschitz continuity of $V$}
\label{LipschitzV}

We assume that a discrete time control problem is given that is stabilizable on all of $\Omega$, i.e. that $S = \Omega$. Furthermore, we require that $f \in C^1(\Omega \times U , \Omega)$, $c \in C^2(\Omega \times U , [0,\infty))$, that $f$
is Lipschitz continuous w.r.t. $x$ with Lipschitz constant $L_f$ and $c$ is Lipschitz continuous w.r.t. $x$ resp. $u$ with Lipschitz constants $L_c$ and $L_u$, resp., and that there is a feedback so that the closed-loop system has 0 as an asymptotically stable fixed point.

The idea of the proof of Lipschitz continuity of $V$ is that the number of time steps needed to steer an arbitrary starting point into a neighborhood of the equilibrium point is bounded. So the proof will consist of two parts, namely 
\begin{enumerate}
\item
finding a neighborhood of the equilibrium point where $V$ is Lipschitz continuous and
\item
using the Lipschitz constants of $f$ and $c$ and the previously mentioned bound on the number of steps needed to control arbitrary initial points into the neighborhood of the equilibrium point and extending the proof of Lipschitz continuity from the neighborhood to the whole state space.
\end{enumerate}


\subsection{Local Lipschitz continuity}

In the following, we will consider the approximation by the following LQR system:
\[
\bar f(x,u) = Ax + Bu, \;\;\; \bar c(x,u) = x^T Q x + u^T R u
\]
where
\[
A = \frac{\partial}{\partial x} f_{(0 , 0)}, \;
B = \frac{\partial}{\partial u} f_{(0 , 0)}, \;
Q = \frac{\partial^2}{\partial x^2} c_{(0 , 0)}, \;
R = \frac{\partial^2}{\partial u^2} c_{(0 , 0)}.
\]
Its optimal feedback is given by a linear map $\bar F(x)= \bar Fx$, where $\bar F$ is a $d \times n$ matrix, see \cite{So98}.
The optimal feedback $\bar F$ for the LQR system is not an optimal feedback for the original system, but a locally stabilizing one; let $\bar V$ be the derived (not optimal) value function of this feedback $\bar F$ for the original system.
Consider the matrix-valued map
\[
M(x,u) = f_x(x,u) + f_u(x,u) \cdot \bar F.
\]
One has for the spectral radius $\rho(M(0,0))<1$ because the original system has 0 as an asymptotically stable fixed point. So there is a norm $\| \cdot \|_a$ on $\R^s$ and an $\eps_a>0$ with $\|M(0,0)\|_a < 1-2 \eps_a$.
Choose open neighborhoods $0 \in \Omega_1 \subset \Omega$, $0 \in U_1 \subset U$ s.t.
\begin{equation}
\label{eq:epsclose}
\|M(x,u)\|_a  \leq 1- \eps_a.
\end{equation}
for $(x,u) \in \Omega_1 \times U_1$.

We choose even smaller open sets $0 \in \Omega_2 \subset \Omega_1, 0 \in U_2 \subset U_1$ by further requiring
$\Omega_2$ to be a sublevel set relative to $\|\cdot\|_a$, $\bar F x \in U_2$ for all $x \in \Omega_2$ and $u + \bar F(\tilde x -x) \in U_1$ for all $x,\tilde x \in \Omega_2, u \in U_2$.

\begin{lemma}
\label{contraction}
Let $(x,u) \in (\Omega_2 \times U_2)$.
Then 
\[
\phi: \Omega_2 \rightarrow \Omega, \quad 
\phi(\tilde x):= f(\tilde x, u+\bar F(\tilde x-x)). \]
is a pseudo-contraction relative to $\|\cdot\|_a$.
\end{lemma}
We use the expression pseudo-contraction to indicate that $\phi$ does not map $\Omega_2$ to itself.
\begin{proof}
One has
\begin{align*}
\left\|
\frac
\partial
{\partial \tilde x}
\phi(\tilde x)
\right\|_a
& =
\|
f_x(\tilde x, u + \bar F(\tilde x-x)) +
f_u(\tilde x, u + \bar F(\tilde x-x)) \cdot \bar F
\|_a \\
& =
\| M(\tilde x, u + \bar F(\tilde x-x)) \|_a
\leq 
1- \eps_a
\end{align*}
by (\ref{eq:epsclose}) and, consequently, by the mean value theorem, $\phi$ is a contraction.
\end{proof}

By setting $(x,u) = (0,0)$ one gets the following corollary. Noting that $\Omega_2$ is a sublevel set relative to $\|\cdot\|_a$, this time the map has images in $\Omega_2$. 

\begin{corollary}
$f(\cdot, \bar F(\cdot))$ is a contraction on $\Omega_2$ relative to $\|\cdot\|_a$.
\end{corollary}
\begin{corollary}
$ \bar V$ is continuous on $\Omega_2$, even Lipschitz continuous.
\end{corollary}
\begin{proof}
$c$ was assumed to be Lipschitz continuous w.r.t. $x$ resp. $u$ with Lipschitz constant $L_c$ resp. $L_u$. By equivalence of norms, $c$ is also Lipschitz continuous w.r.t. $x$ resp. $u$ with Lipschitz constants $L_a$ and $L_{au}$ relative to $\|\cdot\|_a$.
Consequently, $x \mapsto c(x, \bar Fx)$ is Lipschitz continuous with Lipschitz constant $L_a + L_{au} \|\bar F\|_a$ relative to $\|\cdot\|_a$.

It follows from Lemma \ref{contraction} that $\bar V$ is Lipschitz continuous with Lipschitz constant $\frac {L_a + L_{au} \|\bar F\|_a}{\eps_a}$ relative to $\|\cdot\|_a$. This is seen by considering two points $x_0,\tilde x_0 \in \Omega_2$ and comparing their trajectories under the feedback $\bar F$. Their mutual distances relative to $\|\cdot\|_a$ develop at most like a geometric sequence with factor $1-\eps_a$.

\end{proof}


Now we choose an open neighborhood $0 \in \Omega_3 \subset \Omega_2$ 
according to the following lemma.

\begin{lemma}
\label{OptCtrU}
There is a neighborhood $0 \in \Omega_3 \subset \Omega_2$ s.t.
each optimal feedback (of the original system) on $\Omega_3$ lies in $U_2$.
\end{lemma}
\begin{proof}
Because of compactness of $U$ and $c(x,u) > 0$ for $u \neq 0$, one has 
\[
\min_{x \in \Omega, u \notin U_2} c(x,u) = \delta_2 >0.
\]
One can choose $\Omega_3$ s.t.
$\sup_{x \in \Omega_3} \bar V(x) <  \delta_2$,
consequently
\[
\sup_{x \in \Omega_3} V(x) \leq \sup_{x \in \Omega_3} \bar V(x) < \delta_2.
\]
Now, for points in $\Omega_3$, the optimal controls are in $U_2$.
\end{proof}

As an additional condition we require from $\Omega_3$ that $\bar \Omega_3 \subset \Omega_2$. So the Hausdorff distance $d(\Omega_2, \Omega_3)$ is positive.
We choose a neighborhood $0 \in \Omega_4 \subset \Omega_3$ s.t. optimal trajectories starting in $\Omega_4$ stay in $\Omega_3$. 
This is the case if $\sup_{\Omega_4} V \leq \sup_{\Omega_4} \bar V \leq \min_{x \notin \Omega_3, u \in U} c(x,u)$. The last term is positive because of compactness of $\Omega_3^c \times U$.

\begin{lemma}
There is a neighborhood $U_{2 \eps_1}(0) \subset \Omega$ of $0$ where $V$ is Lipschitz continuous.
\end{lemma}
\begin{proof}

For given $x_0,\tilde x_0 \in \Omega_4$ with $\|x_0 - \tilde x_0\| < d(\Omega_2, \Omega_3)$ we choose $(u_k)$ as a nearly optimal control sequence for $x_0$:
$J(x_0, (u_k)) \leq V(x_0) + (L_a + L_{au} \|\bar F\|_a) \frac 1 {\eps_a} \| \tilde x_0-x_0 \|_a$ and such that $(x_k)$ stays in $\Omega_3$. This is possible because $\Omega_3$ is open.
This gives us a sequence $(x_k)$. 
For $\tilde x_k$ we define iteratively $\tilde x_k := f(\tilde x_{k-1} ,\tilde u_{k-1})$ with
\[
\tilde u_k := u_k+\bar F(\tilde x_k-x_k),
\]
i.e. a ``linear correction'' with the information we have from the LQR system.


Consider
\[
\phi_k(\tilde x):= f(\tilde x, u_k+\bar F(\tilde x-x_k)), \]
so by Lemma \ref{contraction}
\[
\|x_{k+1}- \tilde x_{k+1}\|_a 
\leq 
(1-\eps_a) \| x_k- \tilde x_k \|_a
\]
and iteratively on sees that $(\tilde x_k)$ stays in $\Omega_2$ because of the condition $\|x_0 - \tilde x_0\| < d(\Omega_2, \Omega_3)$.


Consequently,
\begin{align*}
|J(x_0,(u_k)) - J(\tilde x_0,(\tilde u_k))| 
&\leq
(L_a + L_{au} \|\bar F\|_a) \frac 1 {1-(1-\eps_a)}\| x_0 - \tilde x_0\|_a \\
& = (L_a + L_{au} \|\bar F\|_a) \frac 1 {\eps_a}\| x_0 - \tilde x_0\|_a
\end{align*}
and thus
\begin{align*}
V(\tilde x_0) & \leq  J(\tilde x_0,(\tilde u_k)) \leq J(x_0,(u_k))+ (L_a + L_{au} \|\bar F\|_a) \frac 1 {\eps_a}\| x_0 - \tilde x_0\|_a \\
& \leq   V(x_0)+ (L_a + L_{au} \|\bar F\|_a) \frac 2 {\eps_a}\| x_0 - \tilde x_0\|_a.
\end{align*}


Changing the roles of $x_0$ and $\tilde x_0$,
and noting that two norms like $\| \cdot \|$ and $\| \cdot \|_a$ on a finite-dimensional space are equivalent, 
we conclude 
\[
|V(x_0) - V(\tilde x_0)| \leq L_{\text{loc}}  \| x_0 - \tilde x_0\|
\]
for some $ L_{\text{loc}} >0$.
Choose $\eps_1>0$ with $U_{2 \eps_1}(0) \subset \Omega_4.$
\end{proof}

\subsection{Global Lipschitz continuity}

\begin{lemma}
$V$ is bounded on $\Omega$.
\end{lemma}
\begin{proof}
Let $x_0 \in \Omega$. Take a stabilizing control sequence $(u_k)_k$. So there is an $k_1$ such that $x_{k_1} \in U_{\eps_1}(0)$. Because of continuity of the system there is a neighborhood $N(x_0)$ such that each $\tilde x \in N(x_0)$ is steered to $U_{2 \eps_1}(0)$ in $k_1$ steps. On $U_{2 \eps_1}(0)$ we already know that $V$ is bounded since $V(x) \leq \bar V(x)$ on $\Omega_3 \supseteq U_{2 \eps_1}(0)$, so it is also bounded on $N(x_0)$ because $c$ is bounded. Now by compactness of $\Omega$, finitely many such sets $N(x_0)$ cover $\Omega$. So $V$ is also bounded on $\Omega$.
\end{proof}

Let $\Delta := \sup_\Omega V$. We assume $c$ to be bounded from below by $\delta_3 > 0$ outside of $U_{\eps_1}(0)$. 
Let $k_0$ be an integer with $k_0 \geq \frac \Delta {\delta_3}$ and $\varepsilon_0 = \eps_1 / L_f^{k_0}$.
The definition of $k_0$ is such that each optimal trajectory reaches $U_{\eps_1}(0)$ in at most $k_0$ steps.

From now on, let $F: \Omega_4 \rightarrow U$ be the (an) optimal feedback which exists on $\Omega_4$ because $V$ is continuous and so the right hand side of the Bellman equation depends continuously on $u \in U$.
\begin{lemma}
Let $x_0 \in \Omega$.
Then there is a neighborhood $x_0 \in A \subset \Omega$ and a constant $L_1 >0$ s.t.
 for any $\tilde x_0 \in A$ we have the following:
Let $(u_k)$ an almost optimal control sequence for $x_0$ (but not for $\tilde x_0$) in the sense that it steers $x_0$ in at most $k_0$ steps to $U_{\eps_1}$ and 
$(\tilde u_k) = (u_0, \dots, u_{k_0}, F(\tilde x_{k_0+1}), F(\tilde x_{k_0+2}), \dots).$
Then 
\[
|J(x_0,(u_k)) - J(\tilde x_0,(\tilde u_k))| \leq L_1 \| x_0 - \tilde x_0\|.
\]
\end{lemma}
\begin{proof}
Again, we consider trajectories $(x_k), (\tilde x_k)$ of $f$ starting at $x_0, \tilde x_0$ with $d := \|x_0 - \tilde x_0\|< \varepsilon_0$ and control sequences $(u_k)$ resp.\ $(\tilde u_k)$. Note that only $(x_k)$ is an almost optimal trajectory.
By optimality of $(u_k)$ and the definition of $k_0$, one has $x_{k_0} \in U_{\eps_1}(0)$.

In addition, one has $\| x_{k_0} - \tilde x_{k_0} \| \leq \varepsilon_0 L_f^n = \eps_1$, so $\tilde x_{k_0} \in U_{2 \eps_1}(0)$. Now,

\begin{align*}
|J(x_0,(u_k)) - J(\tilde x_0,(\tilde u_k))| & \leq |c(x_0,u_0)-c(\tilde x_0,u_0)| + |c(x_1,u_1)-c(\tilde x_1,u_1)| \\
& + \ldots + |c(x_{k_0-1},u_{k_0-1})-c(\tilde x_{k_0-1}, u_{k_0-1})| + |V(x_{k_0}) - V(\tilde x_{k_0})| \\
& \leq L_c d + L_c d L_f + \ldots + L_c d L_f^{n-1} + d L_f^n L_{\text{loc}} \\
& =:   d L_1.
\end{align*}
\end{proof}

\begin{theorem}
Let $V$ be the cost function of a discrete time control problem that is stabilizable on the compact state space $\Omega$, where the dynamical system $f$ is in $C^1(\Omega \times U ,\Omega)$, the cost function $c$ is in $C^2(\Omega \times U ,\R)$, and that has a feedback whose closed-loop system has 0 as an asymptotically stable fixed point.
Then $V$ is Lipschitz continuous. 
\end{theorem}
\begin{proof}
From the setting in the proof of the preceding lemma and by choosing $(u_k)$ appropriately, we get 
\[ 
V(\tilde x_0) \leq J(\tilde x_0,(\tilde u_k)) + L_1 \| x_0 - \tilde x_0\| \leq J(x_0,(u_k)) + 2 L_1 \| x_0 - \tilde x_0\| = V(x_0) + 2 L_1 \| x_0 - \tilde x_0\|.
\]
Changing the roles of $x_0$ and $\tilde x_0$, we conclude 
\[
|V(x_0) - V(\tilde x_0)| \leq  2 L_1 \| x_0 - \tilde x_0\|.
\]
Now, of course, we can skip the assumption $\|x_0 -\tilde x_0\| \leq  \varepsilon_0$, because a local everywhere Lipschitz constant is also a global Lipschitz constant.
\end{proof}

By the definition of $v$ as $v(\cdot)=\exp(-V(\cdot))$, the function $v$ is also Lipschitz continuous, say with Lipschitz constant $L_v$.
In general, $V$ and $v$ can not be expected to be differentiable.

\end{appendix}

\bigskip

\bibliography{references}
\bibliographystyle{abbrv}

\end{document}